\pdfoutput=1
\documentclass[12 pt]{article}
\usepackage{amsfonts, amssymb}
\usepackage{amsthm}
\usepackage{amsmath}
\usepackage{tikz-cd}
\usepackage{mathtools}
\usepackage{cite}
\usepackage{subfig}

\usepackage[utf8]{inputenc}

\usepackage{setspace}

\newtheorem{theorem}{Theorem}

\newtheorem{lemma}[theorem]{Lemma}

\theoremstyle{remark}

\newcommand{\mirror}[1]{{#1}^\updownarrow}
\newcommand{\R}[0]{\mathbb{R}}
\newcommand{\PP}[0]{\mirror{P}}
\newcommand{\A}[0]{\mirror{A}}
\newcommand{\N}[0]{\mathbb{N}}

\newcommand{\im}[0]{\mathrm{Im}}
\newcommand{\Cov}[1]{\mathrm{Cov}(#1)}
\newcommand{\Rng}[1]{\mathrm{Rng}(#1)}

\newcommand{\B}[0]{\mathcal{B}}

\newcommand{\CM}[0]{\mathbb{J}}

\title{Lower bound on the size of a quasirandom forcing set of permutations}
\author{Martin Kure\v{c}ka\footnote{Faculty of Informatics, Masaryk University, Botanick\'a 68A, 602 00 Brno, Czech Republic. E-mail: \texttt{485413@mail.muni.cz}. This research was supported by the MUNI Award in Science and Humanities of the Grant Agency of Masaryk University.
}}
\date{}

\begin{document}

\maketitle

\begin{abstract}
    A set $S$ of permutations is forcing if for any sequence $\{\Pi_i\}_{i \in \N}$ of permutations where the density $d(\pi,\Pi_i)$ converges to $\frac{1}{|\pi|!}$ for every permutation $\pi \in S$, it holds that $\{\Pi_i\}_{i \in \N}$ is quasirandom. Graham asked whether there exists an integer $k$ such that the set of all permutations of order $k$ is forcing; this has been shown to be true for any $k\ge 4$. In particular, the set of all twenty-four permutations of order $4$ is forcing. We provide the first non-trivial lower bound on the size of a forcing set of permutations: every forcing set of permutations (with arbitrary orders) contains at least four permutations.
\end{abstract}

\section{Introduction}
\label{sec:intro}
Random combinatorial structures play an important role in combinatorics and various computer science applications.
If a~large combinatorial structure shares key properties with a~truly random structure it is~said to be \emph{quasirandom}.
The~most studied is~the~theory of~\emph{quasirandom graphs},
which originated in the~seminal works of~R\"odl~\cite{Rod86}, Thomason~\cite{Tho87}, and
Chung, Graham and Wilson~\cite{ChuGW89} in the~1980s.
Graph quasirandomness is~captured by~several seemingly different but in fact equivalent conditions:
the~density of all subgraphs is~close to their expected density in a~random graph,
all but the~largest eigenvalue of the~adjacency matrix are~small,
the~density of a~graph is~uniformly distributed among its (linear size) subset of vertices,
all cuts between linear size subsets of vertices have the~same density, etc.
Besides graphs, there are~results on quasirandomness of many different types of combinatorial structures,
in particular,
tournaments~\cite{BucLSSXX,ChuG91,CorR17,HanKKMPSV19},
hypergraphs~\cite{ChuG90,Gow06,Gow07,HavT89,KohRS02},
set~systems~\cite{ChuG91s},
groups~\cite{Gow08}, and
subsets of integers~\cite{ChuG92}.
In this~paper, we will be concerned with quasirandomness of permutations as studied in~\cite{Coo04,KraP13,ChaKNPSV20}.

One of the~equivalent conditions mentioned above says that a~large graph is~quasirandom if and only if its edge density is~$1/2+o(1)$ and
the~density of cycles of length four is~$1/16+o(1)$. Hence, graph quasirandomness is~captured by~the~density of two specific subgraphs: $K_2$ and $C_4$.
More generally, the~Forcing Conjecture posed by~Conlon, Fox and Sudakov~\cite{ConFS10}
asserts that $C_4$ can be replaced by~any bipartite graph with at least one cycle.
Graham (see~\cite[page 141]{Coo04}) asked whether an analogous result is~true for permutations:
\emph{ Does there exists an integer $k$ such that a~(large) permutation is~quasirandom if and only if
the~density of every $k$-permutation is~$1/k!$\,?}
This~question was answered affirmatively by~Kr\'al' and Pikhurko~\cite{KraP13} by~establishing that any $k\ge 4$ has this~property;
we remark that the~answer is~negative for $k\in\{1,2,3\}$~\cite{CooP08}.
Equivalent results were established in statistics in relation to non-parametric independence tests
by~Yanagimoto~\cite{Yan70}, building on an older work by~Hoeffding~\cite{Hoe48}.
In this~context, we refer the~reader to the~work by~Even-Zohar and Leng~\cite{EveL19} on nearly linear time algorithm for counting small permutation occurrences, which can be used for fast implementation of these tests.

We are~interested in determining the~minimum size of a~set of permutations that captures permutation quasirandomness.
To state our results precisely, we need to introduce some definitions.
A~\emph{permutation of order $n$}, or briefly an \emph{$n$-permutation},
is~a~bijection from $\{1,\ldots,n\}$ to $\{1,\ldots,n\}$;
the~order of a~permutation $\Pi$ is~denoted by~$|\Pi|$.
If $A=\{a_1,\ldots,a_k\}\subseteq \{1,\ldots,n\}$, $a_1<\cdots<a_k$,
then the~subpermutation of $\Pi$ induced by~$A$
is~the~unique permutation $\pi$ of order $|A|=k$ such that
$\pi(i)<\pi(j)$ if and only if $\Pi(a_i)<\Pi(a_j)$.
Subpermutations are~often referred to as \emph{patterns}.
The~\emph{(pattern) density} of a~$k$-permutation $\pi$ in~an~ {$n$\nobreakdash-permutation} $\Pi$
is~the~probability $d(\pi,\Pi)$ that a~randomly chosen $k$-element subset of~$\{1,\ldots,n\}$
induces a~subpermutation equal to $\pi$;
if $k>n$, we set $d(\pi,\Pi)=0$.
We say that a~sequence $\{\Pi_i\}_{i \in \N}$ of~permutations is~\emph{quasirandom}
if for every permutation $\pi$ the~limit of~its~densities in the~sequence $\{\Pi_i\}_{i \in \N}$ converges and satisfies
\begin{equation}
    \label{eq:quasirandom}
    \lim_{i\to\infty}d(\pi,\Pi_i)=\frac{1}{\lvert\pi\rvert!}.
\end{equation}
Finally, we say that the~set $S$ of permutations is~\emph{forcing}
if any~sequence $\{\Pi_i\}_{i\in\N}$ satisfying the~equality \eqref{eq:quasirandom} for all $\pi \in S$ is quasirandom.
In particular, the~results of~\cite{KraP13,Yan70} imply that the~set of all $4$-permutations is~forcing.

A~natural question is~to determine the~minimum size of a~forcing set of permutations.
Inspecting the~proof given in~\cite{KraP13},
Zhang~\cite{Zha} observed that there exists a~$16$-element forcing set of $4$-permutations.
Bergsma~and Dassios~\cite{BerD14} identified an $8$-element forcing sets of $4$-permutations and
Chan et al.~\cite{ChaKNPSV20} found additional three $8$-element forcing sets of $4$\nobreakdash-permutations. In fact, these four $8$-element forcing sets $S$ of $4$-permutations satisfy an~even stronger property,
which is~called \emph{$\Sigma$-forcing},
i.e., a~sequence of permutations is~quasirandom if and only if
the~limit of the~sum of the~pattern densities of permutations in $S$ converges to $\lvert S\rvert/24$.
Our main results asserts that there is~no forcing set containing less than four permutations.
\begin{theorem}
\label{thm:main}
Every forcing set of permutations (of arbitrary, possibly different, orders) has at least four elements.
\end{theorem}
The~proof of Theorem~\ref{thm:main} is~based on analyzing perturbations of a~truly random large permutation.
We present our argument using the~language of the~theory of combinatorial limits, which we briefly introduce in Section \ref{sec:preliminaries}.
In Section~\ref{sec:perturbation},
we establish that the~change of~the~density of~a~pattern after small perturbations
can be described by~a~certain polynomial for~each pattern (the~values of the~polynomial determine
the~gradient of the~density depending on the~location of the~perturbation) and
state a~sufficient condition for being non-forcing in terms of these polynomials.
In Section~\ref{sec:finish},
we show that every set with fewer than four permutations satisfies this~condition with the~exception of~a~few cases.
We then analyze these cases separately to conclude the~proof of Theorem~\ref{thm:main}.

\section{Preliminaries}
\label{sec:preliminaries}
In this section we define notation used in the rest of the paper and present some basic results on permutation limits.
The~set~of~all positive integers is denoted by~$\N$, the~set~of~all~nonnegative integers by $\N_0$, and for any $n\in \N$ the set $\{1, \ldots, n\}$ is denoted by $[n]$. We write $f:[k]\nearrow [n]$ for a non-decreasing function from $[k]$ to $[n]$.



The set of all real matrices of order $n \times m$ is denoted by $\R ^ {n \times m}$. The $i$-th row of a matrix $M$ is~denoted by $r_i(M)$.
A \emph{stochastic} matrix is a non-negative square matrix $M$ such that each of its columns sums to one.
If the same also holds for all its rows, we say that $M$ is \emph{doubly stochastic}. We use $\CM$ to denote the constant doubly stochastic matrix.
The order of $\CM$ will always be clear from context. For a $k$-permutation $\pi$, we define its \emph{permutation matrix} $A_\pi \in \R^{k \times k}$ by setting
\begin{align*}
A_{i,j} = \begin{cases}
    1 & \text{if } \pi(i) = j \text{ and}\\ 
    0 & \text{otherwise.}
\end{cases}
\end{align*}
Note that any permutation matrix is doubly stochastic. By formal linear combination of permutations we mean formal linear combination over real numbers. For any formal linear combination $t_1 \pi_1 + \ldots t_n \pi_n$ of permutations of equal orders, we define its \emph{cover matrix} as \[\Cov{t_1 \pi_1 + \ldots t_n \pi_n} = \sum_{i \in [m]} t_i A_{\pi_i}.\]

A \emph{permuton} is a limit object describing convergent sequences of permutations.
Formally, a~\emph{permuton} $\mu$ is~a~Borel measure on $[0,1]^2$ that has uniform marginals, i.e., postcompositions with both projections are uniform measures.
The notion of~induced subpermutations introduced in Section \ref{sec:intro} can be generalized to any set of points $P = \{(x_1, y_1), \ldots, (x_k, y_k)\}$ such that $x_1 < \ldots < x_k$ and all the $y$-coordinates are pairwise distinct:
for such a set $P$ we call the unique permutation $\pi$ satisfying
\[\pi(i) < \pi(j) \Leftrightarrow y_i < y_j\]
the permutation \emph{induced} by the points $P$. If $k$ points are sampled from $\mu$, they have distinct $x$ and $y$ coordinates with probability one (since $\mu$ has uniform marginals) and therefore they induce a $k$-permutation.
For any $k$\nobreakdash-permutation $\pi$, the probability that a random $k$\nobreakdash-permutation obtained from this sampling is equal to $\pi$ is called the density of $\pi$ in $\mu$ and denoted by $d(\pi, \mu)$.
For example, the uniform Borel measure $\lambda$ on $[0,1]^2$ is a permuton and it holds $d(\pi, \lambda) = \frac{1}{k!}$ for all $k$-permutations $\pi$; in fact, $\lambda$ is the only permuton with this property.

We associate a doubly stochastic square matrix $M$ of order $n$ with a \emph{step permuton} $\mu[M]$ as follows: for a Borel set $X$, the measure of $X$ is 
\begin{align*}
    \mu[M](X) = \sum_{i,j \in [n]}n M_{i,j} \cdot \lambda \left( X \cap \left[\frac{i-1}{n}, \frac{i}{n}\right] \times \left[\frac{j-1}{n}, \frac{j}{n}\right]\right)
\end{align*}
where $\lambda$ is the uniform measure. A straightforward computation leads to an explicit formula for the density of a $k$-permutation $\pi$ in the step permuton $\mu[M]$:

\begin{align}
\label{eq:step-density}
d(\pi, \mu[M]) = \frac{k!}{n^k} \sum_{f,g: [k] \nearrow [n]} \frac{1}{\prod_{i \in [n]} |f^{-1}(i)|!|g^{-1}(i)|!} \times \prod_{m \in [k]} M_{f(m), g(\pi(m))}
\end{align}

Finally, a result from the theory of combinatorial limits yields the following

\begin{lemma}
A nonempty finite set of permutations $S$ is \emph{forcing} if and only if the the uniform permuton is the only permuton $\mu$ satisfying $d(\pi, \mu) = \frac{1}{|\pi|!}$ for any $\pi \in S$.
\end{lemma}

For further details we refer the reader to \cite{KraP13}.

\section{Perturbing the uniform permuton}
\label{sec:perturbation}
In this section, we develop tools for analysing small perturbations of the uniform permuton.
First, we describe a method for perturbing a step permuton and formulate a sufficient condition for a set of permutations to be non-forcing.
Then, we introduce a so-called gradient polynomial which captures the behaviour of perturbations of step permutons
as the order of underlying matrices goes to infinity, and reformulate our sufficient condition in terms of gradient polynomials.
Finally, two different presentations of gradient polynomials are given as they are both needed in specific lemmas.

Fix an integer $n > 1$ and let $k, l \in [n-1]$. We define the matrix $B^{k,l} \in \R^{n \times n}$ by setting
$$
B^{k,l}_{i,j} = \begin{cases}
    1 & \text{if either } i=k \text{ and } j=l \text{, or } i=k+1 \text{ and } j=l+1 \text{,}\\
    -1 & \text{if either } i=k+1 \text{ and } j=l \text{, or } i=k \text{ and } j=l+1 \text{, and}\\
    0 & \text{otherwise.}
\end{cases}
$$
See {Figure~\ref{fig:perturb}(a)} for an example.
Further, for a matrix $x \in \R^{ (n-1) \times (n - 1)}$, we define \[\CM^{x} = \CM + \sum_{1 \leq i,j < n} x_{i,j}B^{i,j}.\]
To simplify the notation we will freely interchange matrices of order ${(n-1) \times (n-1)}$ with vectors of length $(n-1)^2$ obtained by concatenating rows of the matrix.
Note that the matrix $B^{k,l}$ is non-zero only on a $2 \times 2$ submatrix, and the sum of all its rows and columns is zero.
Thus, for any $x \in [-\frac{1}{4n}, \frac{1}{4n}]^{(n-1) \times (n-1)}$, the matrix $\CM^x$ is doubly stochastic, and therefore it gives rise to a step permuton; see Figure~{\ref{fig:perturb}(b)} for an example.
In particular, if $x$ is the zero vector, the permuton $\mu[\CM^x]$ is the uniform permuton.

\begin{figure}
\centering
\begin{tabular}{cc}

\begin{tabular}[t]{c}
\includegraphics[height=4cm]{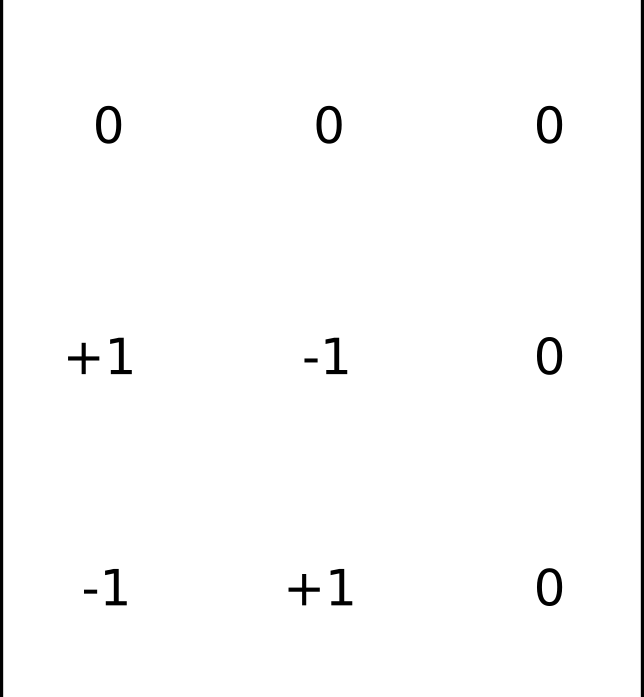}\\
\small{(a)} The matrix $B^{2,1}$.
\end{tabular}&

\begin{tabular}[t]{c}
\includegraphics[height=4cm]{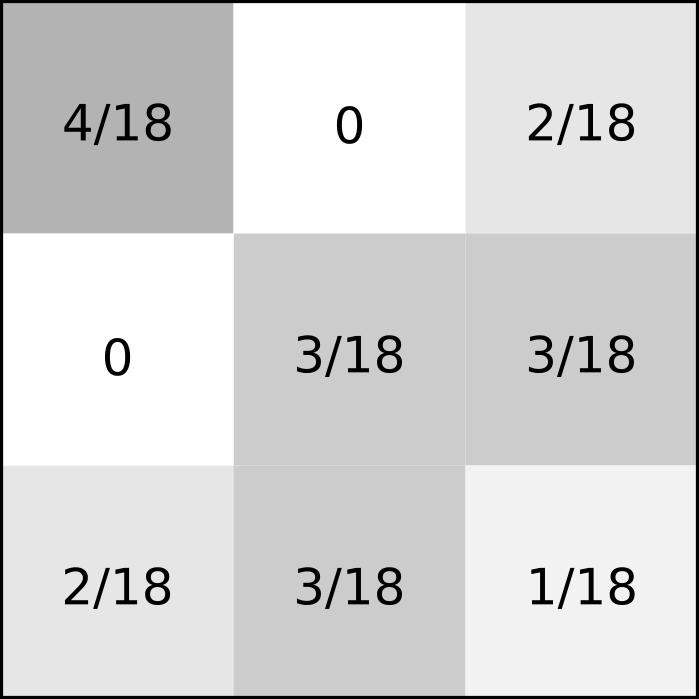}\\
\small{(b)} Measure of individual tiles\\ of the permuton $\mu[\CM^{(1/3,0,0,-1/6)}]$.
\end{tabular}
\end{tabular}

\caption{}
\label{fig:perturb}
\end{figure}

For a permutation $\pi$ we define the \emph{density function} $h_{\pi, n} : \R^{(n-1) \times (n-1)} \rightarrow \R$ where $h_{\pi, n}(x) = d\left(\pi, \mu\left[ \CM^x\right]\right)$.
We wish to analyse permutons $\mu[\CM^x]$ for $x$ close to the zero vector.
In particular, our goal is to find a non-zero $x$ such that the densities of permutations from $S$ in $\mu[\CM^x]$ are the same as in $\mu[\CM]$, i.e., in the uniform permuton.
In the next lemma, we~show that if the gradients of the density functions of the permutations in $S$ satisfy certain conditions,
we are able to find such $x$.

\begin{lemma}
\label{lma:grad}
Let $S$ be a non-empty finite set of permutations. If there exists $n \in \N$ such that $(n-1)^2 > |S|$ and the gradients $\nabla h_{\pi, n}(0,\ldots,0), \pi \in S$, are linearly independent, then $S$ is not forcing.
\end{lemma}
\begin{proof}
Choose indices $i_1, \ldots, i_{|S|}$ such that the gradient vectors $\nabla h_{\pi, n}(0,\ldots,0), \pi \in S$, restricted to these indices are linearly independent.
Note that by the assumption of the lemma, the inequality $(n-1)^2 > |S|$ holds, and so the gradient vectors have at least $|S|+1$ coordinates. 
Let $i_{|S|+1}$ be any index different from $i_1, \ldots, i_{|S|}$.
Define a function $f: \R^{|S|+1} \rightarrow \R^{(n-1) \times (n-1)}$ such that $f(x)_{i_k} = x_k$ for $k \in [|S|+1]$ and $f(x)_i=0$ otherwise.

The gradients ${\nabla (h_{\pi, n}\circ f)(0,\ldots, 0)}$ are also linearly independent, hence we can apply the Implicit Function Theorem for $h_{\pi, n}\circ f$ at the point $(0, \ldots, 0)$.
The theorem yields a~continuous function $g: \R \to \R^{|S|+1}$ defined on $(-\varepsilon, \varepsilon)$ for some $\varepsilon > 0$,
such that $g(r)_{|S|+1} = r$, $h_{\pi, n}(f(g(r))) = h_{\pi, n}(0, \ldots, 0)$ and $\Rng{g} \subseteq [-\frac{1}{4n}, \frac{1}{4n}]^{|S|+1}$.
Recall that $\mu[\CM^x]$ is the uniform permuton if $x = (0, \ldots, 0)$. Therefore $h_{\pi, n}(f(g(x)))$ equals to $\frac{1}{|\pi|!}$ for any ${\pi \in S}$ and any ${x \in (-\varepsilon, \varepsilon)}$.
In particular, $\CM^{g(\varepsilon/2)}$ is a non-uniform permuton that witnesses that $S$ is not forcing.
\end{proof}

As the number of parts of a step permuton increases, the probability that two randomly chosen points share the same part of a step permuton tends to zero. This simplifies the~analysis of gradients significantly and leads us to defining the \textit{gradient polynomial} of~a~permutation $\pi$ --- a limit object which captures the behaviour of the gradients $\nabla h_{\pi,n}$ as $n$ tends infinity. The gradient polynomial $P_\pi(\alpha, \beta)$ is defined as the unique polynomial in two variables which satisfies the equality
\begin{align*}
P_\pi(\alpha, \beta) = \lim_{n \rightarrow \infty} n^3 \frac{\partial}{\partial x_{\lfloor \alpha n \rfloor, \lfloor \beta n \rfloor}}h_{\pi, n}(0, \ldots, 0)
\end{align*}
for any $\alpha, \beta \in (0, 1)$. In the following lemmas we show that the limit always exists, indeed yields a polynomial and we provide an explicit formula for its coefficients.\\

The gradient vector $\nabla h_{\pi, n}(0,\ldots, 0)$ of any permutation $\pi$ can be calculated by a straightforward differentiation of (\ref{eq:step-density}) (see page \pageref{eq:step-density}). In particular the following holds for any positive integer $n > 1$ and $i, j \in [n]$
\begin{align}
\label{eq:grad}
\frac{\partial}{\partial x_{i,j}} h_{\pi, n}(0, \ldots, 0) = \frac{k!}{n^{2k-1}} \sum_{f,g:[k]\nearrow [n]} \frac{1}{\prod_{m\in[n]}|f^{-1}(m)|!|g^{-1}(m)|!} \times \sum_{m \in [k]}  B^{i,j}_{f(m), g(\pi(m))}.
\end{align}

We first show that it is possible to restrict the sum \eqref{eq:grad} to injective functions when considering the limit.

\begin{lemma}
\label{lma:poly-injective-formula}
For any $k$-permutation and $\alpha, \beta \in (0,1)$, the following equality holds if~any of~the~two~limits exists
$$\lim_{n \rightarrow \infty}n^3\frac{\partial}{\partial x_{\lfloor \alpha n \rfloor,\lfloor \beta n \rfloor}} h_{\pi, n}(0, \ldots, 0)
= \lim_{n \rightarrow \infty}\frac{k!}{n^{2k-4}} \sum_{ \substack{f,g:[k]\nearrow [n] \\ f,g \text{ injective}}} \sum_{m \in [k]}  B^{\lfloor \alpha n \rfloor,\lfloor \beta n \rfloor}_{f(m), g(\pi(m))}.$$
\end{lemma}
\begin{proof}
Let $i_n = \lfloor \alpha n \rfloor$, $j_n = \lfloor \beta n \rfloor$, and let $S_n(f, g)$ denote the summand
\[\frac{1}{\prod_{m\in[n]}|f^{-1}(m)|!|g^{-1}(m)|!} \times \sum_{m \in [k]} B^{i_n,j_n}_{f(m), g(\pi(m))}\]
from \eqref{eq:grad}. We first show that the following limit is equal to zero:
\begin{equation}
\label{eq:limit-sum}
L = \lim_{n \rightarrow \infty} \left| \frac{k!}{n^{2k-4}}  \sum_{ \substack{f,g:[k]\nearrow [n] \\ f \text{ or } g \text{ non-injective}}} S_n(f,g) \right|
\end{equation}
Fix $n$.
For each non-injective function $f:[k]\nearrow[n]$ we define a function $\tilde{f}$ as follows:
\begin{align*}
    \tilde{f}(a) = \begin{cases}
    i_n & \text{if } a \in f^{-1}(i_n+1)\\
    i_n+1 & \text{if } a \in f^{-1}(i_n)\\
    f(a) & \text{otherwise}
    \end{cases}
\end{align*}
for any $a \in [k]$. Intuitively, $\tilde{f}$ swaps preimages of $i_n$ and $i_n+1$.
Note that whenever $\im(f)$ does not contain at least one of $i_n$ or $i_n+1$,
then $S_n(f,g) + S_n(\tilde{f}, g) = 0$ for any $g$ since $B^{i_n, j_n}_{i_n,k} = -B^{i_n, j_n}_{i_n+1,k}$ for any $k$. Hence, the sum in \eqref{eq:limit-sum} can be restricted to such terms $S_n(f,g)$ where both $i_n$ and $i_n+1$ are in the range of $f$. A similar argument can be used to exclude those functions $g$ which do not have $j_n$ or $j_n+1$ in their range.

The absolute value of each of the terms $S_n(f, g)$ can be further bounded independently on $n$ by $0 \leq |S_n(f, g)| \leq k$. Therefore, we get

\begin{align*}
0 \leq L = & \lim_{n \rightarrow \infty} \left| \frac{k!}{n^{2k-4}}  \sum_{ \substack{f,g:[k]\nearrow [n] \\
 f \text{ or } g \text{ non-injective}\\
 \lfloor \alpha n \rfloor, \lfloor \alpha n \rfloor+1 \in \im{(f)} \\
 \lfloor \beta n \rfloor, \lfloor \beta n \rfloor+1 \in \im{(g)} \\
 }} S_n(f,g) \right|
 \leq \lim_{n \rightarrow \infty}  \frac{k!}{n^{2k-4}}  \sum_{ \substack{f,g:[k]\nearrow [n] \\
 f \text{ or } g \text{ non-injective}\\
 \lfloor \alpha n \rfloor, \lfloor \alpha n \rfloor+1 \in \im{(f)} \\
 \lfloor \beta n \rfloor, \lfloor \beta n \rfloor+1 \in \im{(g)} \\
 }} \left| S_n(f,g) \right|\\
\leq & \lim_{n \rightarrow \infty} \frac{k \cdot k!}{n^{2k-4}} \left| \left\{ (f,g) \middle| \substack{f,g: [k]\nearrow [n], f \text{ or } g \text{ non-injective}\\
\lfloor \alpha n \rfloor, \lfloor \alpha n \rfloor+1 \in \im{(f)},
\lfloor \beta n \rfloor, \lfloor \beta n \rfloor+1 \in \im{(g)}}
\right\} \right| = \lim_{n \rightarrow \infty} \frac{k \cdot k!}{n^{2k-4}} \, O(n^{2k-5}) = 0.
\end{align*}
Hence, the limit $L$ is indeed equal to zero.
Finally, note that if $f,g:[k]\nearrow[n]$ are injective,
then the product $\prod_{l\in[n]}|f^{-1}(l)|!|g^{-1}(l)|!$ is equal to one.
So if any of the limits from the~statement exist, it holds that
\begin{align*}
    0 &= \lim_{n \rightarrow \infty} \frac{k!}{n^{2k-4}}  \sum_{ \substack{f,g:[k]\nearrow [n] \\ f \text{ or } g \text{ non-injective}}} \frac{1}{\prod_{l\in[n]}|f^{-1}(l)|!|g^{-1}(l)|!} \times \sum_{m \in [k]} B^{\lfloor \alpha n \rfloor,\lfloor \beta n \rfloor}_{f(m), g(\pi(m))}\\
    & = \lim_{n \rightarrow \infty} \left( n^3\frac{\partial}{\partial x_{i,j}} h_{\pi, n}(0, \ldots, 0)
- \frac{k!}{n^{2k-4}} \sum_{ \substack{f,g:[k]\nearrow [n] \\ f,g \text{ injective}}} \sum_{m \in [k]}  B^{\lfloor \alpha n \rfloor,\lfloor \beta n \rfloor}_{f(m), g(\pi(m))} \right)\\
    & = \lim_{n \rightarrow \infty}n^3\frac{\partial}{\partial x_{i,j}} h_{\pi, n}(0, \ldots, 0)
- \lim_{n \rightarrow \infty}\frac{k!}{n^{2k-4}} \sum_{ \substack{f,g:[k]\nearrow [n] \\ f,g \text{ injective}}} \sum_{m \in [k]}  B^{\lfloor \alpha n \rfloor,\lfloor \beta n \rfloor}_{f(m), g(\pi(m))}.
\end{align*}
The statement of the lemma follows.
\end{proof}

Using Lemma \ref{lma:poly-injective-formula}, we find an explicit formula for gradient polynomials. Note that the~formula indeed defines a polynomial.

\begin{lemma}
For any $k$-permutation $\pi$, the gradient polynomial is well-defined and is equal to the following formula:
\begin{align}
\begin{split}
\label{eq:poly-formula}
    P_\pi(\alpha, \beta) = & k! \sum_{m \in [k]}\left(  \frac{k-m}{1 - \alpha} - \frac{m-1}{\alpha} \right)\left( \frac{k-\pi(m)}{1 - \beta} - \frac{\pi(m)-1}{\beta} \right) \\
& \, \frac{\alpha^{m-1}(1 - \alpha)^{k-m}\beta^{\pi(m)-1}(1 - \beta)^{k-\pi(m)}}{(m-1)!(k-m)!(\pi(m)-1)!(k-\pi(m))!}
\end{split}
\end{align}
\end{lemma}

\begin{proof}
Fix $\alpha, \beta \in (0,1)$. We introduce $i_n = \lfloor \alpha n \rfloor$ and $j_n = \lfloor \beta n \rfloor$. We omit the subscript whenever the index is clear from the context. By Lemma \ref{lma:poly-injective-formula}, the following equality holds whenever the right hand side exists

\begin{align*}
P_\pi(\alpha, \beta) = \lim_{n \rightarrow \infty} \frac{k!}{n^{2k-4}} \sum_{ \substack{f,g:[k]\nearrow [n] \\ f,g \text{ injective}}} \sum_{m \in [k]}  B^{i,j}_{f(m), g(\pi(m))}.
\end{align*}
Let $F_n[m \mapsto i]$ denote the number of strictly increasing functions $f:[k] \nearrow [n]$ satisfying $f(m) = i$. We can group the summands by $m$ to obtain
\begin{align*}
& \frac{k!}{n^{2k-4}} \sum_{m \in [k]} \sum_{ \substack{f,g:[k]\nearrow [n] \\ f,g \text{ injective}}}   B^{i,j}_{f(m), g(\pi(m))}\\
= &\frac{k!}{n^{2k-4}} \sum_{m \in [k]} F[m \mapsto i]F[\pi(m)\mapsto j] + F[m \mapsto i+1]F[\pi(m)\mapsto j+1]\\
& - F[m\mapsto i+1]F[\pi(m)\mapsto j] - F[m\mapsto i]F[\pi(m)\mapsto j+1]\\
= &\frac{k!}{n^{2k-4}} \sum_{m \in [k]} \left(F[m \mapsto i]-F[m \mapsto i+1]\right)\left(F[\pi(m)\mapsto j]-F[\pi(m)\mapsto j+1]\right).
\end{align*}
Note that for $k \leq i \leq n-k$ the equality $F[m \mapsto i] = {i-1 \choose m-1}{n-i \choose k-m}$ holds. Then, we can further simplify the sum using the following:
\begin{align}
\label{eq:F}
\begin{split}
  & F[m \mapsto i]-F[m \mapsto i+1] = {i-1 \choose m-1}{n-i \choose k-m}-{i \choose m-1}{n-i-1 \choose k-m} \\
= & \left(\frac{n-i}{n-i-k+m}-\frac{i}{i-m+1}\right){i-1 \choose m-1}{n-i-1 \choose k-m}.
\end{split}
\end{align}

Note that for any $\alpha \in (0,1)$ there exist $N$ such that for any $n \geq N$ it holds $k \leq \lfloor \alpha n \rfloor \leq n-k$ and thus we can always use (\ref{eq:F}) in the following limit:
\begin{align*}
&\lim_{n \rightarrow \infty} \frac{1}{n^{k-2}} \left(F[m \mapsto \lfloor \alpha n \rfloor] - F[m \mapsto \lfloor \alpha n \rfloor + 1] \right)\\
= &\lim_{n \rightarrow \infty}\frac{1}{n^{k-2}} \left(\frac{n-\lfloor \alpha n \rfloor}{n-\lfloor \alpha n \rfloor-k+m}-\frac{\lfloor \alpha n \rfloor}{\lfloor \alpha n \rfloor-m+1}\right){\lfloor \alpha n \rfloor-1 \choose m-1}{{n-\lfloor \alpha n \rfloor-1} \choose k-m}\\
= &\lim_{n \rightarrow \infty}\frac{1}{n^{k-2}} \left(\frac{k-m}{n-\lfloor \alpha n \rfloor-k+m}-\frac{m-1}{\lfloor \alpha n \rfloor-m+1}\right){\lfloor \alpha n \rfloor-1 \choose m-1}{{{n}-\lfloor \alpha n \rfloor-1} \choose k-m}\\
= & \left( \frac{k-m}{1-\alpha} - \frac{m-1}{\alpha} \right)\frac{\alpha^{m-1}(1-\alpha)^{k-m}}{(m-1)!(k-m)!}.
\end{align*}
Similarly we compute

\begin{align*}
&\lim_{n \rightarrow \infty} \frac{1}{n^{k-2}} \left(F[\pi(m) \mapsto \lfloor \beta n \rfloor] - F[\pi(m) \mapsto \lfloor \beta n \rfloor + 1] \right)\\
= &\left( \frac{k-\pi(m)}{1-\beta} - \frac{\pi(m)-1}{\beta}\right)\frac{\beta^{\pi(m)-1}(1-\beta)^{k-\pi(m)}}{(\pi(m)-1)!(k-\pi(m))!}.
\end{align*}
By multiplying these two limits we obtain the equality from the statement of the lemma.
\end{proof}

We next provide an analogy of Lemma \ref{lma:grad} for gradient polynomials.

\begin{lemma}
\label{lma:grad_poly}
Let $S = \{\pi_1, \ldots, \pi_m\}$ be a non-empty set of permutations. If the gradient polynomials $P_\pi , \pi \in S$, are linearly independent, then $S$ is not forcing.
\end{lemma}

\begin{proof}
We prove the contrapositive. Suppose that $S$ is forcing, and for any $n>1$ define $g_i^n = \nabla h_{\pi_i, n}(0,\ldots, 0)$. Lemma \ref{lma:grad} yields that the gradients $g_i^n , i \in [m],$ are linearly dependent for any $n > m + 1$. Therefore, for~any such~$n$ there exists a non-zero tuple of reals $t^n = (t_1^n, \ldots, t_m^n)$ such that $\sum_{i \in {m}} t_i^n g_i^n = (0, \ldots, 0)$.
Moreover, without loss of generality, we can assume $||t^n||_{\infty} = 1$. Since $[-1,1]^m$ is~a~compact set, there exists a convergent subsequence $\{(t_1^{n_j}, \ldots, t_m^{n_j}) \}_{j \in \N}$ converging to~a~non-zero tuple $(t_1, \ldots, t_m)$. Hence for any $\alpha, \beta \in (0,1)$, it holds
\begin{align*}
    \sum_{i \in [m]} t_i P_{\pi_i}(\alpha, \beta) = \sum_{i \in [m]} \lim_{j \rightarrow \infty} t_i^{n_j} \lim_{j \rightarrow \infty} {n_j^3} (g_i^{n_j})_{\lfloor \alpha {n_j} \rfloor, \lfloor \beta {n_j} \rfloor} = \lim_{j \rightarrow \infty} \sum_{i \in [m]} t_i^{n_j} {n_j^3} (g_i^{n_j})_{\lfloor \alpha {n_j} \rfloor, \lfloor \beta {n_j} \rfloor} = 0.
\end{align*}
Therefore the gradient polynomials are linearly dependent since it holds $\sum_{i \in [m]}t_i P_{\pi_i} = 0$.
\end{proof}

For the analysis of gradient polynomials, we use the following kind of vectors. For~an~integer $a\in[k]$ we define a vector $\mathbf{b}^k_a \in \R^k$ as follows:
$$
\left(\mathbf{b}^k_a\right)_i = \begin{cases}
(-1)^{i-1} {a - 1 \choose i-1} & \text{for } 1 \leq i \leq a\\
0 & \text{otherwise}.\\
\end{cases}
$$
For example \[\mathbf{b}_5^5 = (1, -4, 6, -4, 1)^T \text{ and } \mathbf{b}_4^5 = (1, -3, 3, -1, 0)^T.\] We sometimes omit the upper index and write just $\mathbf{b}_a$ when the dimension is clear from~the~context.

Let us denote the linear span of vectors $\mathbf{b}^k_2, \ldots, \mathbf{b}^k_k$ by $\B$ and let $\mathbf{j} = (1,\ldots,1)^T \in \R^k$. Observe that $\B$  is the orthogonal complement of the vector $\mathbf{j}$. Indeed
\[ \langle \mathbf{j} , \mathbf{b}_i \rangle = \sum_{\ell=1}^i (-1)^{\ell-1}{i-1 \choose \ell-1} = 0.\]
Also observe that the all vectors $\mathbf{b}_2^k, \ldots, \mathbf{b}_k^k$ are linearly independent.
In particular, the vectors $\mathbf{j}, \mathbf{b}_2^k, \ldots, \mathbf{b}_k^k$ form a~basis of $\R^k$.\\

The next lemma provides an explicit formula for the coefficients of the gradient polynomials. Let $P(\alpha, \beta)$ be a polynomial in $\alpha, \beta$.
For any $i,j \in \N_0$, we denote the \emph{coefficient} of the monomial $\alpha^i\beta^j$ in $P$ by $c_{i,j}(P)$, i.e., it holds that \[P = \sum_{i,j \in \N_0} c_{i,j}(P)\alpha^i\beta^j.\]

\begin{lemma}
\label{lma:coef}
Let $\pi$ be a $k$-permutation, $P_\pi(\alpha, \beta)$ its gradient polynomial, and $i,j \in \N_0$. Then it holds
\begin{align*}
c_{i,j}(P_\pi) = \frac{k!(-1)^{i+j}}{i!j!(k-i-2)!(k-j-2)!} \left( \mathbf{b}_{i+2}^T \, A_\pi \, \mathbf{b}_{j+2} \right)
\end{align*}
if both $i$ and $j$ are at most $k-2$, and $c_{i,j}(P_\pi)=0$ otherwise.
\end{lemma}
\begin{proof}

For this proof, we set $1/a!= 0$ whenever $a<0$, and ${a \choose b}=0$ whenever $a < b$.
In order to determine the coefficient of $\alpha^i\beta^j$ in $P_\pi$, we need to compute the coefficient of each summand from \eqref{eq:poly-formula} (see page \pageref{eq:poly-formula})and sum them up.
For any $m \in [k]$, the coefficient of $\alpha^i\beta^j$ is the product of the coefficient of $\alpha^i$ in
$\left(\frac{k-m}{1 - \alpha} - \frac{m-1}{\alpha} \right)\frac{\alpha^{m-1}(1 - \alpha)^{k-m}}{(m-1)!(k-m)!}$
and the coefficient of $\beta^j$ in
$\left(\frac{k-\pi(m)}{1 - \beta} - \frac{\pi(m)-1}{\beta} \right)\frac{\beta^{\pi(m)-1}(1 - \beta)^{k-\pi(m)}}{(\pi(m)-1)!(k-\pi(m))!}.$
We first compute the coefficient of $\alpha^i$:
\begin{align*}
    & \left(\frac{k-m}{1 - \alpha} - \frac{m-1}{\alpha} \right)\frac{\alpha^{m-1}(1 - \alpha)^{k-m}}{(m-1)!(k-m)!} \\
    = & \frac{\alpha^{m-1}(1 - \alpha)^{k-m-1}}{(m-1)!(k-m-1)!} - \frac{\alpha^{m-2}(1 - \alpha)^{k-m}}{(m-2)!(k-m)!}  \\
    = & \sum_{\ell=0}^{k-m-1} {k-m-1 \choose \ell} \frac{\alpha^{m-1+\ell}(-1)^\ell}{(m-1)!(k-m-1)!} - \sum_{\ell'=0}^{k-m} {k-m \choose \ell'} \frac{\alpha^{m-2+\ell'}(-1)^{\ell'}}{(m-2)!(k-m)!}.
\end{align*}
The last equality is just an expansion of $(1-\alpha)^{k-m-1}$ and $(1-\alpha)^{k-m}$.
The $i$-th power of $\alpha$ appears for $l = i+1-m$ and $l' = i+2-m$.
This yields that the coefficient of $\alpha^i$ is
\begin{align*}
    &{k-m-1 \choose i+1-m} \frac{(-1)^{i+1-m}}{(m-1)!(k-m-1)!} - {k-m \choose i+2-m} \frac{(-1)^{i-m}}{(m-2)!(k-m)!}\\
    = &{i \choose m-1} \frac{1}{(k-i-2)!i!}(-1)^{i+1-m} + {i \choose m-2} \frac{1}{(k-i-2)!i!}(-1)^{i+1-m}\\
    = &{i+1 \choose m-1} \frac{1}{(k-i-2)!i!}(-1)^{i+1-m}.
\end{align*}
Note that if $m > i+2$, the formula is equal to zero. Similarly, the coefficient of $\beta^j$ in $\left(\frac{k-\pi(m)}{1 - \beta} - \frac{\pi(m)-1}{\beta} \right)\frac{\beta^{\pi(m)-1}(1 - \beta)^{k-\pi(m)}}{(\pi(m)-1)!(k-\pi(m))!}$ is equal to $${j+1 \choose \pi(m)-1} \frac{1}{(k-j-2)!j!}(-1)^{j+1-\pi(m)}.$$
Hence, the coefficient of $\alpha^i\beta^j$ is the following:
\begin{align*}
c_{i,j}(P_\pi) = & k!\sum_{m \in [k]}{i+1 \choose m-1} \frac{(-1)^{i+1-m}}{(k-i-2)!i!}{j+1 \choose \pi(m)-1} \frac{(-1)^{j+1-\pi(m)}}{(k-j-2)!j!}\\
= & \frac{k!(-1)^{i+j}}{i!j!(k-i-2)!(k-j-2)!}\sum_{m \in [k]}(-1)^{m-1}{i+1 \choose m-1}(-1)^{\pi(m)-1}{j+1 \choose \pi(m)-1}\\
= & \frac{k!(-1)^{i+j}}{i!j!(k-i-2)!(k-j-2)!} \left( {\mathbf{b}^T_{i+2}} \, A_\pi \, \mathbf{b}_{j+2} \right).
\end{align*}
\end{proof}
In the following we use
\[K^k_{i,j} = \frac{k!(-1)^{i+j}}{i!j!(k-i-2)!(k-j-2)!},\]
and we omit the upper index when the index is clear from the context. Thus, it holds that
${c_{i,j}(P_\pi) = K_{i,j} \left( \mathbf{b}_{i+2}^T \, A_\pi \, \mathbf{b}_{j+2} \right)}$.\\

Finally we define the \emph{mirror gradient polynomial} $\PP_\pi(\alpha, \beta)$ by setting
\[\PP_\pi(\alpha, \beta) = P_\pi(1-\alpha, \beta).\]
As shown above, the coefficient $c_{i,j}(P_\pi)$ depends on the ``top'' rows of the matrix $A_\pi$, i.e., the first $i+2$ rows.
For any matrix $M \in \R^{n \times m}$ we define its row mirror image $\mirror{M}$ where $\mirror{M}_{i, j} = M_{n-i+1, j}$.
In the next lemma we prove that $\PP_\pi$ behaves in a similar way as $P_\pi$ but its coefficients depend on the ``top'' $i+2$ rows of the matrix $\mirror{A_\pi}$ instead. 
\begin{lemma}
Let $\pi$ be a $k$-permutation, $\PP_\pi(\alpha, \beta)$ its mirror gradient polynomial, and $i,j \in \N_0$. Then the following holds
$$c_{i,j}(\PP_\pi) = \frac{k!(-1)^{i+j+1}}{i!j!(k-i-2)!(k-j-2)!} \left( \mathbf{b}^T_{i+2} \, \A_\pi \, \mathbf{b}_{j+2} \right) = -K_{i,j} \left( {\mathbf{b}^T_{i+2}} \, \A_\pi \, \mathbf{b}_{j+2} \right)$$
if both $i$ and $j$ are at most $k-2$, and $c_{i,j}(\PP_\pi) = 0$ otherwise.
\end{lemma}

\begin{proof}
We perform similar steps as in the previous proof. By the definition of the mirror polynomial, we can substitute $1-\alpha$ into $\ref{eq:poly-formula}$ to obtain

\begin{align*}
\begin{split}
    \mirror{P}_\pi(\alpha, \beta) = & k! \sum_{m \in [k]}\left(  \frac{k-m}{\alpha} - \frac{m-1}{1-\alpha} \right)\left( \frac{k-\pi(m)}{1 - \beta} - \frac{\pi(m)-1}{\beta} \right) \\
& \, \frac{(1-\alpha)^{m-1}\alpha^{k-m}\beta^{\pi(m)-1}(1 - \beta)^{k-\pi(m)}}{(m-1)!(k-m)!(\pi(m)-1)!(k-\pi(m))!}.
\end{split}
\end{align*}

Again, we split each summand into a product of two parts, one depending only  on $\alpha$ and the other on $\beta$.
The part involving $\beta$ is the same as in the previous proof. A straightforward computation analogous to the one in the proof of the previous lemma yields that the coefficient of $\alpha ^ i$ in $\left(\frac{k-m}{\alpha} - \frac{m-1}{1 - \alpha} \right)\frac{(1-\alpha)^{m-1}\alpha^{k-m}}{(m-1)!(k-m)!}$ is
\begin{align*}
{m-1 \choose i+m-k+1} \frac{(-1)^{i+m-k+1}}{(m-1)!(k-m-1)!} - {m-2 \choose i+m-k} \frac{(-1)^{i+m-k}}{(m-2)!(k-m)!}.
\end{align*}
This can be further simplified to
\begin{align*}
{i+1 \choose k-m} \frac{1}{(k-i-2)!i!}(-1)^{i-k+m+1}.
\end{align*}

Hence, the coefficient $c_{i,j}(\PP_\pi)$ is equal to
\begin{align*}
 &\frac{k!(-1)^{i+j}}{i!j!(k-i-2)!(k-j-2)!}\sum_{m \in [k]}(-1)^{m-k+1}{i+1 \choose k-m}(-1)^{\pi(m)+1}{j+1 \choose \pi(m)-1}\\
 &\text{where we can substitute $l = k - m + 1$ and reverse the order of summation to obtain}\\
 &\frac{k!(-1)^{i+j}}{i!j!(k-i-2)!(k-j-2)!}\sum_{\ell \in [k]}(-1)^{-\ell}{i+1 \choose k-(k-\ell+1)}(-1)^{\pi(k-\ell+1)+1}{j+1 \choose \pi(k-\ell+1)-1}\\
 =&\frac{k!(-1)^{i+j+1}}{i!j!(k-i-2)!(k-j-2)!}\sum_{\ell \in [k]}(-1)^{\ell-1}{i+1 \choose \ell-1}(-1)^{\pi(k-\ell+1)-1}{j+1 \choose \pi(k-\ell+1)-1}\\
= & \frac{k!(-1)^{i+j+1}}{i!j!(k-i-2)!(k-j-2)!} \left( {\mathbf{b}^k_{ i+2}}^T \, \A_\pi \, \mathbf{b}^k_{j+2} \right).
\end{align*}

\end{proof}

\section{Sets of linearly dependent polynomials}
\label{sec:finish}
In this section we prove our main result. We call the set $S$ of permutations \textit{linearly dependent} if the gradient polynomials of the permutations in the set $S$ are linearly dependent.
In the previous section, we have shown that any forcing set of permutations is linearly dependent.
We next establish two lemmas that describe general properties of cover matrices of dependent sets of permutations with respect to orders of their permutations.
This will render many triples of permutations to be non-forcing. We then identify
 all linearly dependent sets of size three and prove none of them is forcing.

Recall that the cover matrix of a formal linear combination of $k$-permutations $\omega = {\sum_{i \in [m]} t_i \pi_i}$ is the matrix $\Cov{\omega} = \sum_{i \in [m]} t_i A_{\pi_i}$.
For a dependent set of permutations $S$, the next lemma states a property of  the cover matrix of the permutations with the largest order in $S$.

\begin{lemma}
\label{lma:zero-sums}
Let $\pi_1, \ldots, \pi_m$ be permutations and $t_1, \ldots, t_m$ be reals such that $\sum_{i \in [m]} t_i P_{\pi_i} = 0$ and set $k = max\{|\pi_1|, \ldots, |\pi_m|\}$.
Suppose that $\pi_1, \ldots, \pi_n$ are all the permutations from $S$ with order $k$.
Further, let $2 \leq h \leq k$ be any integer such that the order of all the remaining permutations $\pi_{n+1}, \ldots, \pi_m$ is at most $h-1$.
Let $\omega =  t_1 \pi_1 + \ldots + t_n \pi_n$.
Then the following holds:

\[
\Cov{\omega} \, \mathbf{b}_{h} = (0,\ldots,0)^T \text{ \emph{and}    } \mathbf{b}^T_{h} \, \Cov{\omega} = (0,\ldots,0).
\]
\end{lemma}

\begin{proof}
By Lemma \ref{lma:coef}, the coefficient $c_{i,j}(P_{\pi_\ell})$ is equal to zero for any $\ell > n$ whenever $i$ or $j$ is at least $h-2$.
Therefore, for any $0 \leq i \leq k-2$, we have
\begin{align*}
    0 = \sum_{\ell \in [m]} t_\ell c_{i, h-2}(P_{\pi_\ell}) = \sum_{\ell \in [n]} t_\ell K_{i,h-2} \left( \mathbf{b}_{i+2}^T \, A_{\pi_\ell} \, \mathbf{b}_{h} \right) = K_{i,h-2} \left( \mathbf{b}_{i+2}^T \, \Cov{\omega} \, \mathbf{b}_{h} \right).
\end{align*}
Since $K_{i,h-2}$ is non-zero, it also holds that
\begin{align}
    \label{eq:max-prod-zero}
    0 = \mathbf{b}_{i+2}^T \, \Cov{\omega} \, \mathbf{b}_{h}
\end{align}
implying
\begin{align}
\label{eq:max-prod-zero2}
v^T \, \Cov{\omega} \, \mathbf{b}_{h} = 0
\end{align}
for any $v\in \B$.
Recall that $\B$ is an orthogonal complement of $\mathbf{j}$. Since any vector $u$ that has one entry $-1$, one entry $+1$,
and all the other entries equal to zero belongs to $\B$, it holds ${u \, \Cov{\omega} \, \mathbf{b}_h = 0}$, i.e. $r_p(\Cov{\omega}) \, \mathbf{b}_{h} - r_q(\Cov{\omega}) \, \mathbf{b}_{h} = 0$ for any two rows $r_p(\Cov{\omega})$ and $r_q(\Cov{\omega})$ of matrix $\Cov{\omega}$. This implies $r_p(\Cov{\omega}) \, \mathbf{b}_{h} = r_q(\Cov{\omega}) \, \mathbf{b}_{h}$ for any $p,q \in [k]$.
Therefore, for any $p \in [k]$ the equality $r_p(\Cov{\omega})\, \mathbf{b}_{h} = 0$ holds since
\begin{align*}
    k \, r_p(\Cov{\omega})\, \mathbf{b}_{h}
    = \sum_{\ell \in [k]} r_p(\Cov{\omega})\, \mathbf{b}_{h}
    = \sum_{\ell \in [k]} r_\ell(\Cov{\omega})\, \mathbf{b}_{h}
    = (a, \ldots, a) \, \mathbf{b}_{h} = 0
\end{align*}
where $a$ is the common sum of all the columns.
The first equality of the lemma follows.
The other can be proven by a symmetric argument.
\end{proof}

If all the permutations in a dependent set have the same order, we can prove the following.
\begin{lemma}
\label{lma:const}
Let $\omega = t_1\pi_1 + \ldots + t_m \pi_m$ be a formal linear combination of $k$-permutations. If $\sum_{i \in [m]} t_iP_{\pi_i} = 0$, then the cover matrix $\Cov{\omega}$ is constant.
\end{lemma}

\begin{proof}
We first bound the rank of $\Cov{\omega}$. Recall that the vectors $\mathbf{b}^k_{2}, \ldots, \mathbf{b}^k_{ k}$, and $\mathbf{j}$ form an~orthogonal basis of $\R^k$.
Call that basis $B$.
The matrix $\Cov{\omega}$ is a matrix of some bilinear functional $\phi:\R^k \times \R^k \rightarrow \R$ in the canonical basis.
Let us express the matrix of~the~functional $\phi$ in the basis $B$ by computing the values of $\phi$ on the pairs of basis vectors.
By Lemma \ref{lma:zero-sums}, for $2 \leq i, j \leq k$, it holds
$\mathbf{b}_{ i}^T \, \Cov{\omega} \, \mathbf{b}_{ j} = 0$
implying $\phi(\mathbf{b}_{i}, \mathbf{b}_{j}) = 0$.
By the definition of~a~cover matrix, the sum of any column or row of $\Cov{\omega}$ is equal to a constant $a = \sum_{\ell \in [m]} t_\ell$. Hence, it holds $\mathbf{j} \, \Cov{\omega} \, \mathbf{b}_{ j} = (a,\ldots, a) \, \mathbf{b}_{ j} = 0$ for any $2 \leq j \leq k$ and thus $\phi(\mathbf{j}, \mathbf{b}_{ j}) = 0$.
Similarly, it also holds $\phi(\mathbf{b}_{ i}, \mathbf{j}) = 0$ for any $2 \leq i \leq k$.

The rank of the matrix of $\phi$ in the basis $B$ is at most one since the only nonzero entry it could have is the one corresponding to the value $\phi(\mathbf{j}, \mathbf{j})$.
The change of the basis does not change the rank of the matrix of a functional, therefore, the rank of $\Cov{\omega}$ depends only on the value of $\phi(\mathbf{j}, \mathbf{j})$. In particular, it is either zero or one. If it is zero, then the matrix $\Cov{\omega}$ is the constant zero matrix. In the latter case, the columns of $\Cov{\omega}$ are multiple of each other and since they have constant non-zero sum, they are all equal. Similarly, all~the~rows of $\Cov{\omega}$ are equal. The fact that $\Cov{\omega}$ is constant follows.
\end{proof}


In the next lemma we prove that if there exists a formal linear combination of gradient polynomials equal to zero but having all coefficients non-zero, then it contains at least two permutation with the maximum order.

\begin{lemma}
\label{lma:max2} Let $\pi_1, \ldots, \pi_m$ be permutations of order at least two and suppose $|\pi_1| \geq \ldots \geq |\pi_m| $. If there exist non-zero reals $t_1, \ldots, t_m$ satisfying $\sum_{i \in [m]} t_i P_{\pi_i} = 0$, then $m \geq 2$ and $|\pi_1|=|\pi_2|$.
\end{lemma}
\begin{proof}
Suppose for a contradiction that $\pi_1$ is the unique permutation among $\pi_1,\ldots, \pi_m$ with the largest order (in particular this holds if $m=1$). Then, the~cover matrix $\Cov{t_1 \pi_1} = t_1 A_{\pi_1}$ contains exactly one non-zero element in each row, and, therefore, the~product of any row with the vector $\mathbf{b}_k$ from Lemma \ref{lma:zero-sums} is non-zero. This contradicts Lemma~\ref{lma:zero-sums}. Hence, there is at least one permutation with order $|\pi_1|$ other than $\pi_1$. In particular $m \geq 2$.
\end{proof}

The next lemma combines Lemma \ref{lma:const} and Lemma \ref{lma:max2} to exclude most of the sets of two or three permutations of equal orders from being linearly dependent. 

\begin{lemma}
\label{lma:same-sizes}
Let $S$ be a linearly dependent set of permutations whose orders are larger than one. Then:
\item a) $S$ is not a singleton.
\item b) If $|S|=2$, then both permutations in $S$ have order two.
\item c) If $|S|=3$ and all permutations in $S$ have the same order, then their common order is three.
\end{lemma}

\begin{proof}
Let $\pi_1, \ldots, \pi_m$ be permutations in $S$. By Lemma \ref{lma:max2}, it holds that ${m = |S| \geq 2}$. Since $S$ is linearly dependent, there exists a non-zero tuple of reals $(t_1, \ldots, t_m)$ such that $\sum_{i \in [m]} t_i P_i = 0$.
Let $\omega$ denote the formal linear combination $\sum_{i \in [m]} t_i \pi_i$.
First note that regardless whether $m=2$ or $3$, we may assume that all permutation in $S$ have the same order and all the coefficients $t_i, i \in [m],$ are non-zero. Indeed for $m=2$, both statements follows as a consequence of Lemma \ref{lma:max2}. For $m=3$, the equality of orders follows by the assumption of the Lemma. Furthermore, observe that if any of the coefficients $t_i$ was equal to zero, we would proceed as in the part b) and show that two of the permutations in $S$ have order two. By the assumption, the third should have the same order which is impossible since there are only two distinct permutations of order two.

If the order of permutations in $S$ is larger than $m$, then there exists a zero entry in the matrix $\Cov{\omega}$. Since $m$ is at most three, there exists $i\in [m]$ and $\pi_j\in S$ such that $\pi_j(i)$ differs from all the other permutations from $S$ evaluated at $i$, i.e., $\pi_j(i) \neq \pi_{j'}(i)$ for $j' \neq j$. Otherwise, all permutations would be identical. Hence, the matrix $\Cov{\omega}$ has a non-zero entry, specifically $\Cov{\omega}_{i, \pi_j(i)} = t_j$. In particular, the matrix $\Cov{\omega}$ is not constant, which contradicts Lemma \ref{lma:const}. Therefore, all the permutations in the set $S$ have order $m$ regardless whether $m=2$ or $m=3$.
\end{proof}

In the next lemma we exclude all the set of three permutation containing a ``large'' permutations from being linearly dependent.

\begin{lemma}
\label{lma:gt4}
Let $S = \{\pi_1, \pi_2, \pi_3\}$ be a linearly dependent set of non-trivial permutations. If $|\pi_1| \geq |\pi_2| \geq |\pi_3|$ and $|\pi_1| > 3$, then $|\pi_2| = |\pi_3| = 2$.
\end{lemma}

\begin{proof}
Let $(t_1, t_2, t_3)$ be a non-zero tuple of reals such that $\sum_{i \in [3]} t_iP_{\pi_i} = 0$. If $t_1=0$, the statement follows from Lemma \ref{lma:same-sizes}. Hence, we can assume $t_1\neq0$.
By Lemma \ref{lma:max2}, the orders of $\pi_1$ and $\pi_2$ are equal and $t_2 \neq 0$. By assumption, the order of $\pi_1$ is at least four, therefore Lemma \ref{lma:same-sizes} implies that the orders of permutations in $S$ cannot be equal to each other, i.e., $|\pi_1| = |\pi_2| > |\pi_3|$.

Let $k$ denote the order of the permutations $\pi_1$ and $\pi_2$, and let $\omega$ denote the formal linear combination $t_1\pi_1 + t_2\pi_2$.
We first show that the absolute values of the coefficients $t_1$ and $t_2$ are equal.
Indeed, if there exists $i$ such that $\pi_1(i) = \pi_2(i)$, then the following holds by Lemma \ref{lma:zero-sums}
\[(t_1 + t_2) {k-1 \choose \pi_1(i)-1} (-1)^{\pi_1(i)-1} = r_i\left(\Cov{\omega}\right)\,\mathbf{b}_k = 0,\]
hence $t_1$ is equal to $-t_2$.
Otherwise, the values $\pi_1(i)$ and $\pi_2(i)$ are different for every $i \in [k]$.
In particular, we can without loss of generality assume $|t_1| \leq |t_2|$ and let $i$ be such that $\pi_1(i)$ equals one. Then, it holds that
\[t_1{k-1 \choose 0} + t_2{k-1 \choose \pi_2(i)-1}(-1)^{\pi_2(i)-1} = 0,\]
but that is possible only if $\pi_2(i)=k$ and $|t_1|=|t_2|$.

We next show that $|\pi_3| = k-1$. Suppose that this is not the case, i.e., $|\pi_3| < k - 1$. Let $i$ be such that $\pi_1(i) \neq \pi_2(i)$, i.e., the row $r_i(\Cov{\omega})$ contains exactly two non-zero entries.
By Lemma \ref{lma:zero-sums}, the following equalities hold
\begin{align*}
r_i(\Cov{\omega}) \, \mathbf{b}_{k} & = 0\\
r_i(\Cov{\omega}) \, \mathbf{b}_{k-1} & = 0.
\end{align*}
The first equality implies that $\pi_1(i) = k + 1 - \pi_2(i)$ since $|t_1|=|t_2|$, while the second implies that $\pi_1(i) = k - \pi_2(i)$ which is impossible.

Note that the equality $\sum_{i \in [3]}t_i P_{\pi_i} = 0$ holds if and only if the equality $\sum_{i \in [3]}t_i \PP_{\pi_i} = 0$ holds.
Lemma \ref{lma:zero-sums} yields that the cover matrix $\Cov{\omega}$ is symmetric up to the sign, i.e., ${\Cov{\omega}} = (-1)^{k}\mirror{\Cov{\omega}}$ (recall that $\Cov{\omega}$ has at most two nonzero entries in each column and $|t_1| = |t_2|$). Let $0 \leq i, j \leq k-3$. It follows that
\begin{align*}
&t_3 c_{i,j}(\PP_{\pi_3})\\
= &-t_2 c_{i,j}(\PP_{\pi_2}) - t_1 c_{i,j}(\PP_{\pi_1}) \\
= & K_{i,j} \left(\mathbf{b}_{ i+2}^T \, \mirror{\Cov{\omega}} \, \mathbf{b}_{ j+2} \right)\\
= & (-1)^{k} K_{i,j}\left(\mathbf{b}_{ i+2}^T \, \Cov{\omega} \, \mathbf{b}_{ j+2}\right)\\
= &(-1)^{k} t_2 c_{i,j}(P_{\pi_2}) + (-1)^{k} t_1 c_{i,j}(P_{\pi_1}) \\
= & (-1)^{k+1} t_3 c_{i, j}(P_{\pi_3}).
\end{align*}
Since $t_3$ is non-zero, it holds that
\begin{align*}
(\mathbf{b}^{|\pi_3|}_{ i+2})^T \, \A_{\pi_3} \, \mathbf{b}^{|\pi_3|}_{j+2} = - c_{i,j}(\PP_{\pi_3}) \left(K^{|\pi_3|}_{i,j}\right)^{-1}
= (-1)^{k} c_{i, j}(P_{\pi_3}) \left(K^{|\pi_3|}_{i,j}\right)^{-1} = (-1)^{k}(\mathbf{b}^{|\pi_3|}_{ i+2})^T \, A_{\pi_3} \, \mathbf{b}^{|\pi_3|}_{j+2}
\end{align*}
for any $0 \leq i,j \leq k-3$.
We conclude that the equality $u^T \, \A_{\pi_3} \, v = (-1)^{k} u^T \, A_{\pi_3} \, v$ holds
for any vectors $u,v \in \B$.
Choose $i,j$ such that $\pi_3(i) = 1$ and $\pi_3(j)=2$.
Define vectors $u,v \in \B$ by setting $u_i = v_1 = 1, u_j = v_2=-1$ and $u_k = v_k = 0$ for any other $k$. The product $u^T \, A_{\pi_3}\, v$ is equal to two but since $|\pi_3| = k-1 > 2$, the absolute value of the product $(-1)^{k} u^T \, \A_{\pi_3} \, v$ is at most one which is a contradiction.
\end{proof}

\begin{figure}[ht!]
      \includegraphics[clip, width=0.25\textwidth]{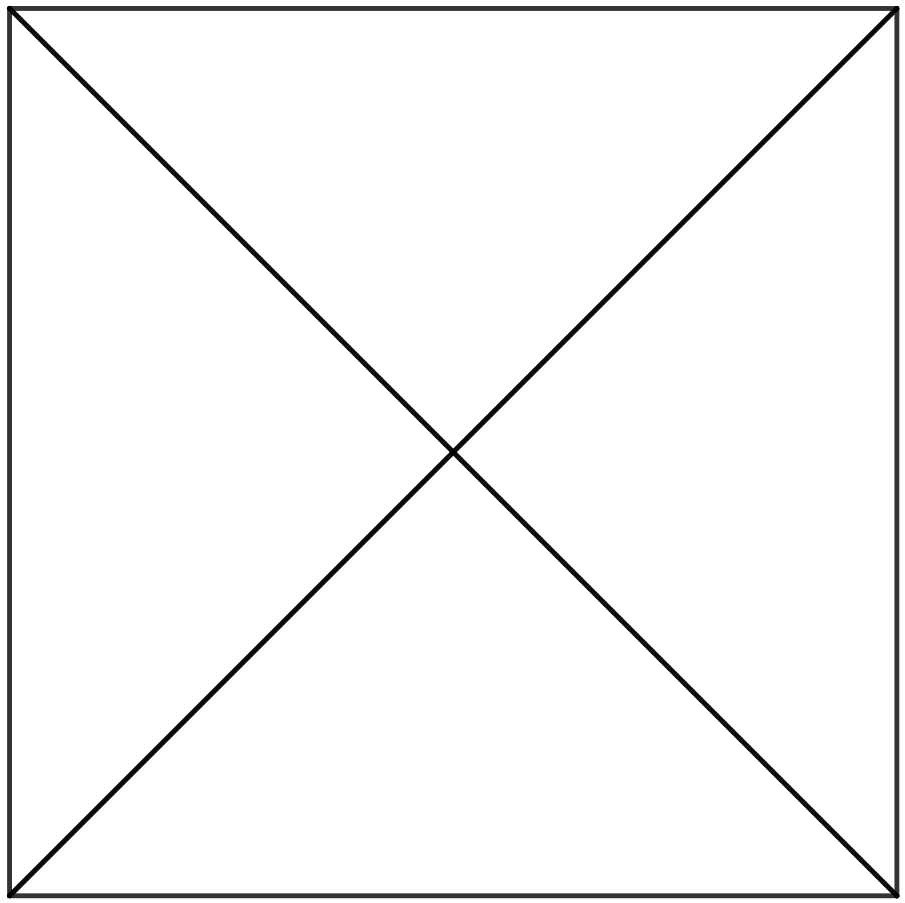}
\hspace{\fill}
      \includegraphics[clip, width=0.25\textwidth]{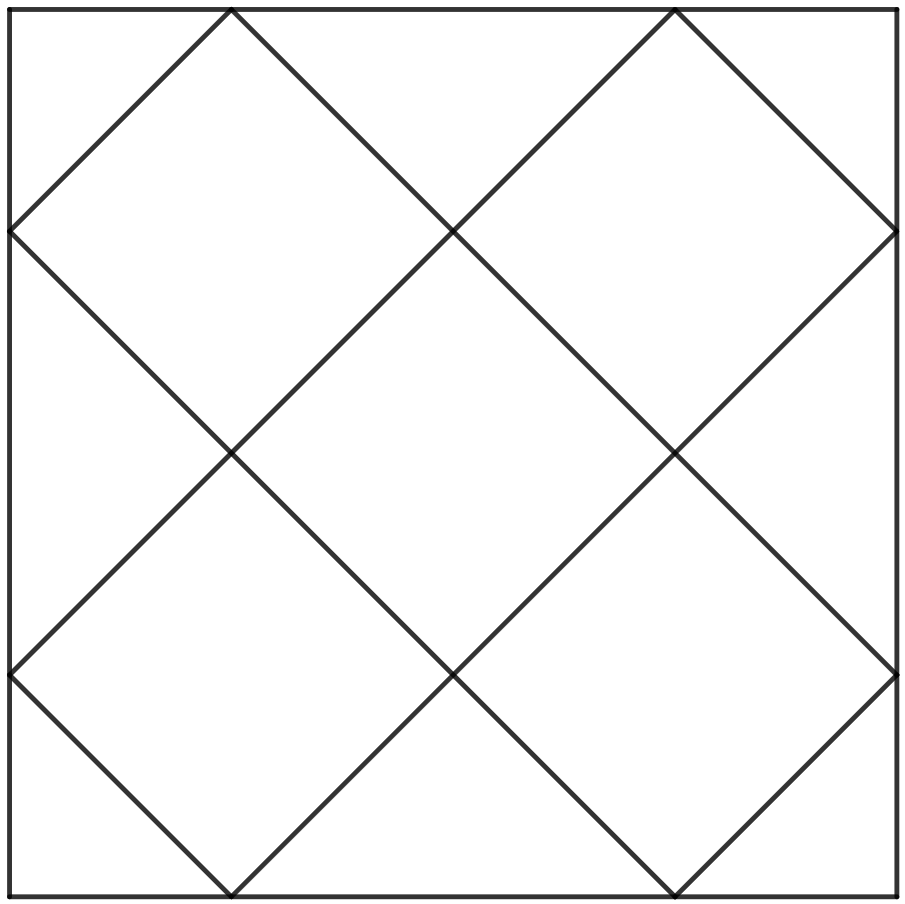}
\hspace{\fill}
      \includegraphics[clip, width=0.25\textwidth]{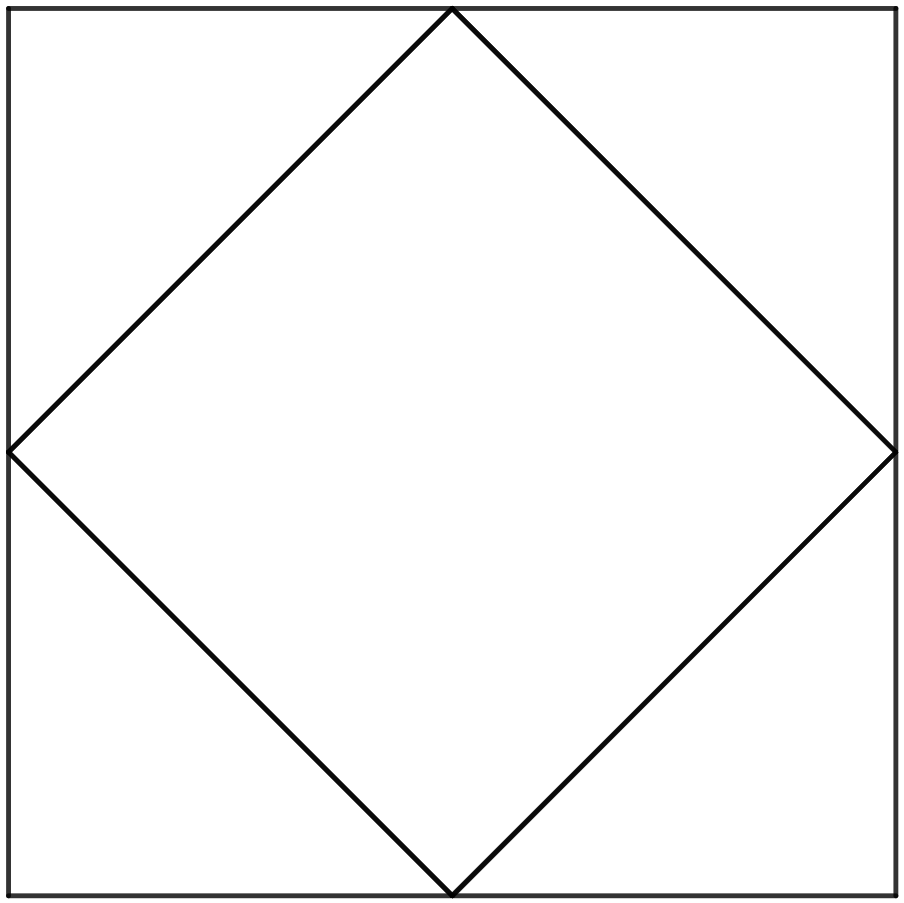}\\
\caption{\label{permutons}The sets $M_0$, $M_{0.5}$ and $M_1$}
\end{figure}

The next lemma provides the last ingredient to prove Theorem \ref{thm:main}. The lemma can be found for instance in \cite{KraP13} but we include a sketch of the proof for completeness.

\begin{lemma}
\label{lma:small}
There exists a non-uniform permuton $\mu$ such that for any $k$-permutation $\pi$ with $k < 4$ it holds that $d(\pi, \mu) = \frac{1}{|\pi|!}$.
\end{lemma}
For any $\alpha \in [0,1]$ define $M_\alpha$ to be the set of all the points $(x, y) \in [0, 1]^2$
such that $x+y \in  \{1-\frac{\alpha}{2}, 1+\frac{\alpha}{2}, \frac{\alpha}{2}, 2-\frac{\alpha}{2}\}$ or $y-x \in \{-\frac{\alpha}{2}, \frac{\alpha}{2}, 1-\frac{\alpha}{2}, \frac{\alpha}{2}-1\}$.
See the illustration in Figure~\ref{permutons}.
Let $\mu_{\alpha}$ be a permuton that is obtained by uniformly distributing the mass along~$M_\alpha$.
Note that $\mu_\alpha$ is invariant under horizontal and vertical reflection, and, therefore, the density of both $12$ and $21$ in $\mu_\alpha$ is equal to $1/2$ for any $\alpha$.
A simple calculation yields that $d(123, \mu_0) = 1/4$ and $d(123, \mu_1) = 1/8$. Since $d(123, \mu_{\alpha})$ is a continuous function there exists $\alpha_0 \in (0, 1)$ such that $d(123, \mu_{\alpha_0}) = 1/6$. The symmetries of the permuton imply that $d(123, \mu_{\alpha_0}) = d(321, \mu_{\alpha_0})$ and $d(132, \mu_{\alpha_0}) = d(312, \mu_{\alpha_0}) = d(213, \mu_{\alpha_0}) = d(231, \mu_{\alpha_0})$. In addition, the sum of these six densities is one, hence all six densities are equal to $1/6$.\\

We are finally ready to prove Theorem \ref{thm:main}:
\begin{proof}[Proof of Theorem \ref{thm:main}]
Let $S=\{\pi_1, \pi_2, \pi_3\}$ be a forcing set consisting of three permutations and suppose $|\pi_1| \geq |\pi_2| \geq |\pi_3|$. Note that we can assume without loss of generality that all the permutations have order at least two. Lemma \ref{lma:small} asserts that there is no forcing set of permutations of order at most three, hence we can further assume $|\pi_1| > 3$.

By Lemma \ref{lma:grad_poly}, the set $S$ is linearly dependent. Lemma \ref{lma:gt4} yields that the order of both permutations $\pi_2, \pi_3$ is two, i.e., $S=\{\pi_1, 12, 21\}$. Since $S$ is forcing, the set $\{\pi_1, 12 \}$ is also forcing. However, the set $\{\pi_1, 12 \}$ cannot be forcing by Lemma \ref{lma:grad_poly} and \ref{lma:max2}.
We conclude that there does not exist any forcing set consisting of three permutations.
\end{proof}





\section*{Acknowledgements}
The author is grateful to Jake Cooper for the careful reading of the manuscript. Special thanks are due to Dan Kr\'al' for overall guidance and the amount of time spent on consulting the problem.

\bibliography{main}{}
\bibliographystyle{plain}
\end{document}